\documentclass[12pt]{article}
\usepackage{epsfig}
\usepackage{enumerate}
\usepackage{fancyhdr}
\usepackage[mathscr]{eucal}
\usepackage{amssymb}

\usepackage{graphicx}
\usepackage{color}

\newcommand\C{{\cal C}}
\newcommand\E{{\cal E}}

\newcommand\PG{{\rm{PG}}}

\newcommand\GF{{\rm{GF}}}

\newcommand\D{{\cal D}}

\renewcommand{\P}{\mathcal{P}}

\newcommand\bs{\backslash}
\newcommand\st{:}
\newcommand\A{{\cal A}}

\renewcommand\O{{\cal O}}
\newcommand{\T}{{\mathcal T}}
\newcommand{\V}{{\mathcal V}}

\newcommand{\linfty}{\ell_{\infty}}
\newcommand{\li}{\ell_{\infty}}

\renewcommand{\H}{{\mathcal H}}

\newcommand\U{{\cal U}}
\newcommand\Q{{\cal Q}}
\newcommand\sinfty{{\Sigma_\infty}}
\newcommand\si{{\Sigma_\infty}}

\renewcommand\S{{\cal S}}
\newcommand\Cp{{\mathcal O}}

\newcommand\cpl{$\C$-plane}
\newcommand\infpt{$\infty$-point}
\newcommand\cpt{$\C$-point}
\newcommand\cln{$\C$-line}
\newcommand\ti{t_\infty}

\newcommand{\Label}{\label}

\newtheorem{theorem}{Theorem}[section]
\newtheorem{lemma}[theorem]{Lemma}
\newtheorem{corollary}[theorem]{Corollary}
\newtheorem{defn}[theorem]{Definition}

\newenvironment{proof}{\noindent{\bf Proof}\hspace{0.5em}}
    { \null  \hfill $\square$ \par}

\newcommand{\R}{{\cal R}}

\begin{document}
\title{Characterising pointsets in $\PG(4,q)$ that correspond to conics}

\author{S.G. Barwick and Wen-Ai Jackson
\date{}\\ 
School of Mathematics, University of Adelaide\\
Adelaide 5005, Australia
}

\maketitle

AMS code: 51E20.

Keywords: finite projective geometry, Bruck-Bose representation, conics, characterization

\begin{abstract}
We consider a non-degenerate conic in $\PG(2,q^2)$, $q$ odd, that is tangent to $\ell_\infty$ and
look at its structure in
the Bruck-Bose representation in $\PG(4,q)$.
We determine which combinatorial properties of this set of points in
$\PG(4,q)$ are
needed to reconstruct the conic in $\PG(2,q^2)$. That is,
we define a set $\C$ in $\PG(4,q)$ with $q^2$ points that satisfies certain
combinatorial
properties. We then show that if $q\ge 7$, we can use $\C$ to construct a regular
spread $\S$ in the hyperplane at infinity of $\PG(4,q)$, and that $\C$ corresponds to a conic in the
Desarguesian plane $\P(\S)\cong\PG(2,q^2)$ constructed via the Bruck-Bose correspondence.

\end{abstract}


\section{Introduction}

We begin by describing the Bruck-Bose representation of a translation plane of order $q^2$ with kernel containing
$\GF(q)$ in $\PG(4,q)$, 
proved independently by Andr\'e \cite{andr54} and Bruck and Bose
\cite{bruc64,bruc66}. 
Let $\sinfty$ be a hyperplane of $\PG(4,q)$ and let $\S$ be a spread of
$\sinfty$.
Consider the incidence structure whose {\em points} are the
points of $\PG(4,q)\setminus\sinfty$, whose {\em lines} are the planes of
PG$(4,q)$ which do not lie in $\Sigma_\infty$ but which meet
$\Sigma_\infty$ in a line of $\S$ and where {\it incidence\/} is inclusion.
This incidence structure is an affine translation plane and can be
uniquely completed to a projective translation plane $\P(\S)$ of order $q^2$ 
by adjoining the line at infinity $\linfty$ whose points are the
elements of the spread $\S$.  The line $\ell_\infty$ is a translation
line for $\P(\S)$.  The translation plane $\P(\S)$ is
Desarguesian if
and only if the spread $\S$ is regular (\cite{bruc64}).
For more details on the Bruck-Bose representation, see~\cite{barw08}, in particular, note that Baer subplanes of $\PG(2,q^2)$ secant to $\li$ are in one to one correspondence with affine planes of $\PG(4,q)$ that do not contain a line of the regular spread $\S$.  

In $\PG(4,q)$, with  $\sinfty$ being the hyperplane at infinity, we call the points of $\PG(4,q)\setminus\sinfty$  {\em
  affine points} and the points in $\si$ {\em infinite points}.
The lines and planes of $\PG(4,q)$ that are not contained in $\sinfty$ are called {\em
  affine lines} and {\em affine planes} respectively. 

Now consider a non-degenerate conic $\overline\C$ in 
 $\PG(2,q^2)$ that is tangent to $\li$ in the point $P_\infty$. In the
Bruck-Bose representation of $\PG(2,q^2)$ in $\PG(4,q)$, the affine points
$\C=\overline\C\setminus\{P_\infty\}$ correspond to a set of $q^2$ affine points in $\PG(4,q)$
also denoted by $\C$. These
points satisfy a variety of properties. 
In \cite{barwcaps}, algebraic properties of these points were determined, and it was shown they were caps. In this article we are interested in combinatorial properties the points satisfy, in particular, combinatorial properties relating to the planes of $\PG(4,q)$. Our aim was to
find the smallest set of these properties that would allow us to
reconstruct the projective plane and the conic $\overline\C$. The properties of the conic in $\PG(2,q^2)$ that we  are interested in are given in the next lemma, which is proved in Section~\ref{section:lemma-proof}. 

\newcommand\oc{\overline\C}
\begin{lemma}\Label{conic-satisfies-props}
 Let $\overline\C$ be a non-degenerate conic of $\PG(2,q^2)$, $q$ odd, 
that is tangent to $\li$ in a point $P_\infty$. Define a $\C$-plane to 
be a Baer subplane of $\PG(2,q^2)$ that is secant to $\li$ and meets 
$\C=\overline\C\setminus\{P_\infty\}$ in $q$ points. Then
\begin{enumerate}
\item Each $\C$-plane meets $\C$ in a $q$-arc. Further, if a Baer 
subplane secant to $\li$
   meets $\C$ in more than four points, it is a $\C$-plane.
\item  Every pair of points of $\C$ lie in exactly one $\C$-plane.
\item The affine points of $\PG(2,q^2)$ are of three types: points of $\C$;
   points on no $\C$-plane (the interior points of $\overline\C$); and points on
   exactly two $\C$-planes (the  exterior points of $\overline\C$).
\end{enumerate}
\end{lemma}

We now consider what these properties correspond to in the Bruck-Bose representation in $\PG(4,q)$. 
The points of $\C$ correspond to affine points of $\PG(4,q)$, and the \cpl s correspond to affine planes of $\PG(4,q)$ that contain $q$ points of $\C$. We suppose that we have a set of affine points and $\C$-planes satisfying the $\PG(4,q)$ equivalence of the combinatorial properties of Lemma~\ref{conic-satisfies-props}, and show that we can reconstruct the conic in the Bruck-Bose plane.
The main result of this paper is the characterisation given in the next theorem.

\begin{theorem}\Label{mainthm} Let $\sinfty$ be the hyperplane at infinity in $\PG(4,q)$,
  $q\ge7$, $q$ odd.
Let $\C$ be a set of $q^2$ affine points, called $\C$-points, and suppose there exists a set of
affine planes called \cpl s satisfying the following properties:
\begin{itemize}
\item[{\rm (A1)}] Each \cpl\ meets $\C$ in a $q$-arc. Further, if a plane
  meets $\C$ in more than four points, it is a \cpl.
\item[{\rm (A2)}]  Every pair of points of $\C$ lie in exactly one \cpl.
\item[{\rm (A3)}] The affine points of $\PG(4,q)$ are of three types: points of $\C$,
  points on no \cpl, and points on
  exactly two \cpl s.  
\end{itemize}
Then there exists a unique spread $\S$ in $\sinfty$ so that the \cpt s in the Bruck-Bose plane $\P(\S)$ form a $q^2$-arc of $\P(\S)$.
Moreover, the spread $\S$ is regular, and so $\P(\S)\cong\PG(2,q^2)$, and
the $q^2$-arc can be completed to a conic of $\PG(2,q^2)$. 
\end{theorem}

We note that a similar characterisation when $q$ is even is given in \cite{conicqeven}.

The rest of this paper is devoted to proving this theorem. The main structure of the proof is as follows. 
We need a number of preliminary results, leading to  Theorem~\ref{lines-form-a-spread} where we show how to construct a spread $\S$ of $\si$ from $\C$.  In Corollary~\ref{cor:regular-spread}, we show that $\S$ is a regular
spread.  In
Theorem~\ref{thm:isanarc}, we show that the set $\C$ corresponds to a $q^2$-arc in the Bruck-Bose plane $\P(\S)\cong\PG(2,q^2)$, and hence as $q$ is odd, $\C$ is contained in a unique conic.
In Theorem~\ref{thm:diffspread}, we show that if $\S'$ is any other spread
of $\si$, then the points of $\C$ will not correspond to an arc in the associated Bruck-Bose plane $\P(\S')$.

\section{Proof of Lemma~\ref{conic-satisfies-props}}\Label{section:lemma-proof}

In this section we prove that a conic in $\PG(2,q^2)$ satisfies the combinatorial properties stated in Lemma~\ref{conic-satisfies-props}.

\begin{proof}{\bf of Lemma~\ref{conic-satisfies-props}}
%
Let $\overline\C$ be a non-degenerate conic in $\PG(2,q^2)$ tangent to $\li$ in the point $P_\infty$, and let $\C=\overline\C\setminus\{ P_\infty\}$. We begin with a note about subconics of $\overline\C$ in Baer subplanes. 
Suppose $\pi$ is  a Baer subplane that meets $\overline\C$ in a subconic $\O$. Note that for each point $P$ in $\O\cap\pi$, the tangent line of $\overline\C$ at $P$ is a line of $\pi$. Conversely, if a line $\ell$ of $\pi$ is a tangent line of $\overline\C$, then the point of contact $\ell\cap\overline\C$ lies in $\pi$. 

We now prove part 1. The $q$ points of $\C$ in a $\C$-plane lie on the conic $\overline\C$, so are clearly a $q$-arc.
Suppose $\pi$ is a Baer subplane secant to $\li$ which meets $\C$
in five points, then those five points define a unique conic in $\PG(2,q^2)$, namely $\overline\C$. Further, the five points define a unique conic $\O$ in $\pi$, which is necessarily a subconic of $\overline\C$.  As $\pi$ is secant to $\li$, it contains the tangent $\li$ to $\oc$, hence by the above note, it contains the point $P_\infty$ of contact of the tangent $\li$ to $\oc$. Hence $\pi$ meets $\C$ in $q$ points, and so is a $\C$-plane.

For part 2, let $P,Q$ be two points of $\C$ with tangents to $\oc$ 
labeled $t_P$, $t_Q$  respectively.  Suppose $\pi$ is any Baer  subplane secant to $\li$ containing $P,Q$. As $\pi$ contains $\li$, it contains $P_\infty$ by the above note. Further, as $\pi$ contains $P,Q$, we have that $t_P,t_Q$ are lines of $\pi$.   Thus $\pi$ contains the quadrangle  $P,Q,t_P\cap \li,P_\infty$ and hence there is exactly one such subplane. However,  the quadrangle  $P,Q,t_P\cap \li,P_\infty$  defines a unique Baer subplane  which contains five elements from the conic $\oc$: namely the points $P,Q,P_\infty$ and the tangents $\li,t_P$, and thus  contains a subconic of $\oc$ containing $P_\infty$. Thus every pair of points of $\C$ lie on exactly one \cpl.

Note that this allows us to 
 count the number of $\C$-planes. Let $\A$ be 
the incidence structure with {\sl points} the points of $\C$;
  {\sl lines} the $\C$-planes; and inherited incidence. Then $\A$ is a 
$2$-$(q^2,q,1)$ design, and so is an
affine plane of order $q$. Hence there are $q^2+q$ $\C$-planes.

To prove part 3, let $\pi$ be a $\C$-plane, and let $X$ be an affine point of $\pi$ 
that is not in $\C$. Then $X$ lies on two tangents of $\pi\cap\oc$, which 
lie either in $\pi$ or in $\PG(2,q^2)\bs\pi$, depending on whether $X$ is an exterior 
point or an interior point of the conic $\pi\cap\oc$. In either case, $X$ 
lies on two tangents of $\oc$ in $\PG(2,q^2)$, and so is an exterior 
point of $\oc$. Hence the affine points in a $\C$-plane either lie in 
$\C$, or are exterior points of $\oc$. That is, the interior points of 
$\oc$ lie on zero $\C$-planes. It is straightforward to show that  the 
group of homographies of $\PG(2,q^2)$ fixing $\overline\C$ and $\li$ is 
transitive on the affine exterior points of $\overline\C$. Hence all the 
affine exterior points of $\overline\C$ lie on a common number $x$ of 
$\C$-planes. We now count incident pairs $(X,\pi)$ where $X$ is an affine
exterior point of $\overline\C$ that lies on a $\C$-plane $\pi$. We have 
$q^2(q^2-1)/2\times x=(q^2+q)(q^2-q)$, hence $x=2$ as required.
\end{proof}

\section{Proof of Theorem~\ref{mainthm}}
\subsection{Properties of $\C$-planes}

Let $\C$ be a set of $q^2$ affine points in $\PG(4,q)\setminus\si$, $q\ge7$, $q$ odd, satisfying
assumptions (A1), (A2) and (A3) of Theorem~\ref{mainthm}. 
We begin by noting that the $\C$-points and  \cpl s form an
affine plane. This gives us a natural set of parallel classes which will be useful
in our proof.

\begin{lemma}\Label{affine-plane}
Let $\A$ be the incidence structure with {\sl points} the points of $\C\!$;
 {\sl lines} the \cpl s; and inherited incidence. Then $\A$ is an
affine plane of order $q$, and so there are $q^2+q$ \cpl s. 
Hence there are $q+1$ {\sl parallel classes} of $\C$-planes, each containing $q$ {\em parallel} \cpl s.
\end{lemma}

\begin{proof}
By (A1) and (A2), $\A$ is a 2-$(q^2,q,1)$ design and so is
an affine plane of order $q$. 
\end{proof}
 
Note that in $\PG(4,q)$, two $\C$-planes $\pi,\alpha$  meet in a point  or a line. The parallel classes of the affine plane $\A$ tell us whether the intersection $\pi\cap\alpha$ contains a $\C$-point. So we have: if $\pi$ and $\alpha$ belong to the same parallel class, then they have no common $\C$-point; whereas if $\pi$ and $\alpha$ belong to different parallel classes, then they share exactly one $\C$-point.

%


Let $\pi$ be a \cpl. By (A1), $\pi$ contains $q$ \cpt s that form a
$q$-arc. As $q$ is odd, this arc uniquely completes to a conic by the addition of a point which we
denote $\pi_\infty$ (see~\cite[Theorem 10.28]{hirs98}).  We call $\pi_\infty$ the {\em $\infty$-point} of $\pi$ since we will show in Lemma~\ref{cplane-meet-theorem} that
$\pi_\infty$ is on the line $\pi\cap\sinfty$, and so is an infinite point
of $\PG(4,q)$. That is:

\begin{defn} Let $\pi$ be a $\C$-plane, then the {\em $\infty$-point} $\pi_\infty$ of $\pi$ is the point that uniquely completes the $q$-arc $\pi\cap\C$ to a conic. 
\end{defn}

To study the intersection of $\C$-planes in more detail, we will need the next lemma involving three $\C$-planes that all contain  a fixed $\infty$-point. 

\begin{lemma}\Label{one-parallel-class}
Let $\pi$ be a \cpl\ with $\infty$-point $\pi_\infty$. Suppose two \cpl s $\alpha,\beta$ both meet
$\pi$ in distinct  lines through $\pi_\infty$.  Then $\alpha$ and $\beta$ share no point other than $\pi_\infty$. 
\end{lemma}

\begin{proof}
Let $\pi,\alpha,\beta$ be $\C$-planes such that $\alpha\cap\pi$, $\beta\cap\pi$ are distinct lines through $\pi_\infty$.
Note that if $\pi_\infty\not\in\Sigma_\infty$, then $\pi_\infty$ is on three
\cpl s, namely $\pi$, $\alpha$ and $\beta$, contradicting (A3).  So we have $\pi_\infty\in\Sigma_\infty$.  
Suppose $\alpha$ and $\beta$ meet in a line $\ell$ (so $\ell$ contains
$\pi_\infty$). The line $\ell$ contains at most one \cpt\ by
(A2).
Note that the $q+1$ lines of $\pi$ through $\pi_\infty$ consist of $q$
1-secants to $\C$, and one 0-secant to $\C$. Hence at least one of the
 lines $\alpha\cap\pi$, $\beta\cap\pi$  is a 1-secant of $\C$. Without loss of generality, suppose $\pi\cap\beta$ meets $\C$
in the point $B$. As $\C$ meets $\pi$ in a $q$-arc, there are two 1-secants to $\C$ through $B$ in $\pi$, one is
$B\pi_\infty$, let the other meet $\pi\cap\alpha$ in the point $R$. In the plane
$\beta$, there are two 1-secants to $\C$ through $B$, one is
$B\pi_\infty$, let the other meet $\beta\cap\alpha$ in the point $S$.

We now show that there exists a 2-secant $m$ of $\alpha\cap\C$ which does not contain the
points $S, R$ or $\pi_\infty$, and which does not meet $\ell$ or $\alpha\cap\pi$
in a \cpt.
As $q\geq7$, the $q$-arc $\alpha\cap\C$ has at least five points
$X,Y,Z,V,W$ which do not lie on $\ell$ or on $\alpha\cap\pi$ (each of which contain at most one $\C$-point). A 2-secant
through two of these points will meet $\ell$ and $\alpha\cap\pi$ in points
that are not in $\C$. 
First consider the three 2-secants $XY,XZ,YZ$. If none of these is
suitable, then each contains one of $R$, $S$, 
$\pi_\infty$. So next consider the three 2-secants $XW,YW,ZW$. If none of
these is suitable, then each contains one of $R,S,\pi_\infty$. Without loss
of generality suppose $R\in XW$, $S\in YW$ and $\pi_\infty\in ZW$. 
Hence we also have $R\in YZ$, $S\in XZ$, and $\pi_\infty\in XY$.
Now consider the 2-secant $VX$.  
If $S\in VX$, then 
$X,Z,V$ would be collinear, contradicting $\alpha\cap\C$ being an arc. 
So $S\notin VX$. 
 Similarly $R\not\in VX$ and $\pi_\infty\not\in VX$.  So $m=VX$ is a
 2-secant of $\C$ not containing $R,S$ or $\pi_\infty$, and meeting $\ell$ and $\alpha\cap\pi$ in 
 points not in $\C$. 

So there exists a line $m$ in $\alpha$ that contains two $\C$-points, but does not contain the points $S, R$ or $\pi_\infty$, and  does not meet $\ell$ or $\alpha\cap\pi$ in a \cpt.
  Consider the plane  $\tau=\langle m,B\rangle$.  Now, $\tau$ meets $\beta$ in a line
through $B$ that is not a 1-secant of $\C$ as
$\pi_\infty,S\notin\tau$. Hence $\tau$ contains two points of
$\beta\cap\C$ (one of which is $B$). Similarly, $\tau$ meets $\pi$ in two
points of $\pi\cap\C$, one of which is $B$. Hence $\tau$ contains five
distinct points of $\C$, and so is a \cpl\ by (A1). Hence $m$ contains
two $\C$-points which lie on two \cpl s, namely $\alpha$ and $\tau$, 
contradicting (A2).  Hence $\alpha$ and $\beta$ cannot meet in a line $\ell$, so they share no point other than $\pi_\infty$.
\end{proof}

The next lemma is key to our proof. We first
show that for each \cpl\ $\pi$, the  $\infty$-point $\pi_\infty$ of
the $q$-arc  $\pi\cap\C$ lies in $\si$. Then we study how
\cpl s can meet, and investigate properties of the parallel classes of the affine plane $\A$ defined in Lemma~\ref{affine-plane}.
We use the following definition.

\begin{defn}
If $\pi$ is a \cpl, then
we call the line $\pi\cap\sinfty$ a {\em \cln}. 
\end{defn}

\begin{lemma}\Label{cplane-meet-theorem}
\begin{enumerate}
\item If $\pi$ is a $\C$-plane, then its $\infty$-point $\pi_\infty$ lies in $\si$. 
\item Let $\pi,\alpha$ be distinct $\C$-planes.
\begin{enumerate}
\item If $\pi,\alpha$ lie in the same parallel class, then they meet in exactly one point $\pi_\infty$ (which is equal to $\alpha_\infty$).
\item If $\pi,\alpha$ lie in different parallel classes, then they either meet in exactly one point of $\C$, or they meet in an affine line through $\pi_\infty$ 
 (which is equal to $\alpha_\infty$) that contains
  one point of $\C$.
\end{enumerate}
\item The \cpl s  which are in a common parallel class all share
  the same \infpt, and they pairwise intersect in only this
  \infpt.  Moreover, every \infpt\ defines exactly two
  parallel classes.
\end{enumerate}
\end{lemma}

\begin{proof}
Let $\pi$ be a fixed \cpl\ and let $\pi_\C=\pi\cap\C$. We say a \cpl\ $\alpha$ (distinct from $\pi$) {\em covers} an affine point $P$ of
$\pi\bs\pi_\C$ if it contains $P$. 
Further an affine line of $\pi$ is called a {\em cover line} of $\pi$ if it is
contained in a \cpl\ distinct from $\pi$.

We work with the parallel classes of the affine plane $\A$ defined in Lemma~\ref{affine-plane}. The $q$ points of $\pi_\C$ each lie on $q$ further $\C$-planes (one in each of the parallel classes not containing $\pi$). The $q^2-q$ affine points of  $\pi\bs\pi_\C$ each lie in exactly one further $\C$-plane by (A3). 
That is, each of the $q^2-q$ affine points of  $\pi\bs\pi_\C$ is covered by  exactly one $\C$-plane. We will investigate how the $\C$-planes cover these points. First we look at how another $\C$-plane $\alpha$ meets $\pi$. Now $\alpha\cap\pi$ is either a point or a line of $\PG(4,q)$. If $\alpha,\pi$ are in different parallel classes, then $\alpha,\pi$ contain exactly one common $\C$-point, so $\alpha\cap\pi$ is either a $\C$-point, or an affine 1-secant of $\C$. If $\alpha,\pi$ are in the same parallel class, then $\alpha,\pi$ contain no common $\C$-point.
Let $${\mathscr P}=\{\beta_1,\ldots,\beta_{q-1}\}$$ be the $\C$-planes in the same parallel class as $\pi$. Then we have:
\begin{itemize}\item[I.] 
\begin{itemize}
\item[(a)] If $\alpha\in{\mathscr P}$, then $\alpha$ meets $\pi$ in either a point of $\pi\setminus\pi_\C$ (possibly in $\si$); or an affine 0-secant of $\pi_\C$; or an infinite line.
\item[(b)] If $\alpha\not\in{\mathscr P}$, then $\alpha$ meets $\pi$ in either exactly one $\C$-point, or in an affine 1-secant of $\pi_\C$.
\end{itemize}
\end{itemize}

We now focus on I(a) and show that no \cpl\ in $\mathscr P$ meets $\pi$
in an affine line.  Suppose there is one such plane
$\beta\in\mathscr P$
meeting $\pi$ in the affine line $b$, so $b$ is a 0-secant of $\pi_\C$ (we work towards a contradiction to show that $b$ cannot exist). 
No other cover line in $\pi$ can meet $b$ in an affine point, otherwise
that point would be on three \cpl s, contradicting (A3). Hence every other cover line of $\pi$  (if
any) meets $b$ in the point $X=b\cap\si$. 

Let $\mathcal L$ be the set of $q$ affine lines of $\pi$ through $X$. There are four possibilities for these lines, depending on the location of the $\infty$-point $\pi_\infty$. Firstly, suppose $\pi_\infty\in\si$. Note that $\pi_\infty$ lies on exactly one 0-secant of $\pi_\C$, which in this case is  $\pi\cap\si$. Hence $\pi_\infty\neq X$, as otherwise $X$ lies on two 0-secants of $\pi_\C$, namely $\pi\cap\si$ and $b$. Thus the $q$ lines of $\mathcal L$ consist of one 1-secant, 
 $\frac{q-1}2$ 2-secants,
 and $\frac{q-1}2$ 0-secants of $\pi_\C$.  Secondly, if 
 $\pi_\infty\not\in\sinfty$ and $X\pi_\infty$ is a 0-secant of $\pi_\C$, then $\mathcal L$ contains one 1-secant, $\frac{q-1}2$
 2-secants and $\frac{q-1}2$ 0-secants of $\pi_\C$. Thirdly, if $\pi_\infty\not\in\sinfty$ and $X\pi_\infty$
 is a 1-secant of $\pi_\C$, then either $\mathcal L$ contains three 1-secants, $\frac{q-3}2$  2-secants and $\frac{q-3}2$ 0-secants of $\pi_\C$, or $\mathcal L$ contains one 1-secant, $\frac{q-1}2$  2-secants and $\frac{q-1}2$ 0-secants of $\pi_\C$.

Let $\mathcal L'\subset\mathcal L$ be the set of 2-secants of $\pi_\C$ through $X$. In each of the four cases, $\mathcal L'$ contains either $\frac{q-1}2$ or  $\frac{q-3}2$ lines. As $q\geq7$, $\mathcal L'$ contains at least two lines $\ell,m$. By I, the affine cover lines of $\pi$ are either 0-secants or 1-secants of $\pi_\C$, so $\ell$ and $m$ are not cover lines. Further, we noted above that all cover lines must contain $X$. Hence the affine points of $\ell,m$ not in $\C$ are not contained in any cover lines. Thus by I,  the affine points of $\ell,m$ not in $\C$ must be covered by planes in ${\mathscr P}\setminus \{\beta\}$, with each plane covering at most one affine point. 
However, there are $2(q-2)$ such affine points and only $q-2$ \cpl s in ${\mathscr P}\setminus \{\beta\}$, a contradiction. 
Hence the line  
$b$ does not exist, so no \cpl\ in $\mathscr
P$ meets $\pi$ in an affine line. Hence the following modification of I(a) holds:
\begin{itemize}\item[I(a)$'$] If $\alpha\in{\mathscr P}$, then  $\alpha$ meets $\pi$ in either  a point of $\pi\setminus\pi_\C$, or in an infinite line. 
\end{itemize}

We now show that $\pi_\infty\in\si$. By I(a)$'$, \cpl s in $\mathscr P$ meet $\pi$ in either 0 or 1 affine
points. 
Hence the \cpl s
in $\mathscr P$ 
cover at most $q-1$ of the $q^2-q$ affine points of
$\pi\setminus\pi_\C$. Recall that each of these affine points is covered by exactly one $\C$-plane. 
By I(b), \cpl s not in ${\mathscr P}$ either meet $\pi$ in exactly a \cpt\ or in a 1-secant of $\pi_\C$.
So there are at least
$((q^2-q)-(q-1))/(q-1)=q-1$ \cpl s not
in $\mathscr P$ that meet $\pi$ in
a cover line, denote them by  $$\mathscr
Q=\{\alpha_1,\ldots,\alpha_{q-1}\}.$$  
For each $i$, 
$\alpha_i$ and $\pi$ have exactly one common $\C$-point, denoted $A_i$. 
As the cover line $\alpha_i\cap\pi$ is a 1-secant of the $q$-arc $\pi_\C$ in $\pi$, it is either the line $A_i\pi_\infty$, or it is the unique tangent line at $A_i$  to the conic $\pi_\C\cup\pi_\infty$, denote this tangent line by $\ell_i$. Suppose one of the cover lines $\alpha_1\cap\pi$ is the unique tangent line $\ell_1$ (we work towards a contradiction to show this cannot happen). Note that two cover lines cannot meet in an affine point of $\pi\setminus\pi_\C$ by (A3). 
As the cover lines  $\alpha_2\cap\pi,\ldots,\alpha_{q-1}\cap\pi$ are all 1-secants, they either meet $\ell_1$ in $A_1$ (there is at most one other such 1-secant, namely $A_1\pi_\infty$) or in a 1-secant through $\ell_1\cap\si$ (there are at most two other such 1-secants through $\ell_1\cap\si$). This gives at most three other possibilities for cover lines, contradicting $|\mathscr Q\setminus\alpha_1|=(q-1)-1\geq 5$. Thus no tangent line $\ell_i$ is a cover line, so $A_i\pi_\infty$ is a cover line for $i=1,\ldots,q-1$.  
That is,  $\pi_\infty$ lies on $q$ $\C$-planes, so by (A3), $\pi_\infty$ lies in $\pi\cap\si$. This completes part 1.

We now further improve I(a)$'$ and show that the $\C$-planes in $\mathscr P$ either meet $\pi$ in the infinite point $\pi_\infty$, or in an infinite line. First note that we have shown that the $q-1$ \cpl s in $\mathscr Q$ all contain $\pi_\infty$. Hence by
Lemma~\ref{one-parallel-class}, these planes  pairwise meet only in
$\pi_\infty$. In particular, they pairwise have no common $\C$-point. Hence the {\em planes in ${\mathscr Q}$ all belong to the same parallel class}, of which there is one more member, $\alpha_q$ say. Denote by $\mathscr X$ the set of  $q-1$ affine points of $\pi\bs\pi_\C$
that are not covered by the \cpl s in $\mathscr Q$.
The points of $\mathscr X$ lie on a line through $\pi_\infty$. We will show that $\alpha_q$ meets
$\pi$ in this line. 

Consider a plane $\alpha_1\in\mathscr Q$. A similar
argument to that of part 1 shows that the $\infty$-point $(\alpha_1)_\infty\in\alpha_1\cap\si$. If
$\alpha_1$ and $\pi$ share \cln s, that is 
$\alpha_1\cap\si=\pi\cap\si$, then $\alpha_1,\pi$ must coincide since they share two lines, namely their \cln\ and an affine line.  Hence their \cln s are distinct, and a similar argument to that of part 1 shows that $\pi$ meets $\alpha_1$ in a line through $(\alpha_1)_\infty$, hence $(\alpha_1)_\infty=\pi_\infty$. Next, similar to
the set $\mathscr Q$ related to the \cpl\ $\pi$, consider the 
set $\mathscr T$ of at least $q-1$ \cpl s that each meet $\alpha_1$ in a line through $(\alpha_1)_\infty$. By the above argument, they lie in a common parallel class. As $\pi\in\mathscr T$, it follows that the parallel class containing $\mathscr T$ is the
parallel class $\mathscr P\cup\pi$ (since every \cpl\ belongs to exactly one parallel class). 
Hence the (at least) $q-2$ planes $\beta_i\in\mathscr P\cap\mathscr T$
satisfy $(\beta_i)_\infty=(\alpha_1)_\infty=\pi_\infty$. Thus at least $q-2$ of the 
\cpl s in $\mathscr P$ meet $\pi$ exactly in the point $\pi_\infty$. 
In particular, at least $q-2$ of the $\beta_i$ contain no affine points of $\pi$. Hence there are at
least $q-2$ affine
points in the set $\mathscr X$ that are not covered by planes in $\mathscr
P$. Hence by I, there is a $\C$-plane $\alpha$ not in $\mathscr P$ that meets $\pi$ in a 1-secant of $\pi_\C$ through $\pi_\infty$. This line covers $q-1$ affine points, so covers all the points of $\mathscr X$. Note that by Lemma~\ref{one-parallel-class}, $\alpha$  belongs to the
parallel class containing $\mathscr Q$, so $\alpha=\alpha_q$. In summary, we have:
\begin{enumerate}\item[II.] 
\begin{itemize}
\item $\C$-planes in the parallel class ${\mathscr P}\cup {\pi}$ all contain $\pi_\infty$, and pairwise contain no common affine point;
\item 
Each $\C$-plane in the parallel class ${\mathscr Q}\cup\alpha_q$ meets $\pi$ in a 1-secant through $\pi_\infty$;
\item the remaining $\C$-planes each meet $\pi$ in exactly one $\C$-point.
\end{itemize}
\end{enumerate}

To complete the proof of part 2, we need to show that $\C$-planes in $\mathscr P$ meet $\pi$ in exactly $\pi_\infty$. 
Suppose not, that is, by II, suppose that there is a $\C$-plane $\beta_1$ in $\mathscr P$ that meets $\pi$ in an infinite line, so $\beta_1\cap\pi=\pi\cap\si$. Let $\alpha\in\mathscr Q\cup\alpha_q$, so $\alpha$ meets $\pi$ in a 1-secant through $\pi_\infty$. As $\pi,\beta_1$ are in the same parallel class, $\beta_1$ and $\alpha$ lie in different parallel classes, so have a common $\C$-point, $A$ say. Then the line $A\pi_\infty$ lies in both $\beta$ and $\alpha$, contradicting 
 Lemma~\ref{one-parallel-class}.
 Hence $\C$-planes in the parallel class $\mathscr P\cup\pi$ pairwise meet in exactly the point $\pi_\infty$. 
  
Hence we have shown that \cpl s in the same parallel class 
have the same $\infty$-point, and pairwise meet in exactly this point. To prove the second statement of part 3, we 
note that by II, the $\C$-planes in the parallel classes $\mathscr P\cup\pi$ and $\mathscr Q\cup\alpha_q$ all contain the point $\pi_\infty$. Further, the remaining $\C$-planes do not contain $\pi_\infty$. So $\pi_\infty$ lies on the $\C$-planes of precisely two parallel classes.
\end{proof}

This lemma shows that two $\C$-planes cannot meet in a $\C$-line, so we have:

\begin{corollary}\Label{count-cline} Each $\C$-plane has a unique $\C$-line, so there are $q^2+q$ distinct $\C$-lines in $\si$.
\end{corollary}

We now characterise the points of $\Sigma_\infty$ in relation to the
\cln s.

\begin{defn} A point of $\Sigma_\infty$ that lies on no \cln\ is called a
{\em 0-point}.
\end{defn}

\begin{lemma}\Label{lem:pttype}
There are three types of points in  $\Sigma_\infty$: ${q+1\over 2}$
\infpt s, ${q+1\over 2}$ 0-points, and $q^3+q^2$ points which lie on
exactly one \cln\ each.
\end{lemma}

\begin{proof} 
Let $\ell,m$ be two $\C$-lines in $\si$, and let $\alpha,\beta$ be the unique $\C$-planes with $\C$-lines $\ell,m$ respectively. By Lemma~\ref{cplane-meet-theorem}(2), 
$\ell,m$ meet in a exactly one point (namely $\alpha_\infty=\beta_\infty$), or not at all. So each point of $\si$ is either a 0-point, an $\infty$-point, or lies on exactly one $\C$-line. By Corollary~\ref{count-cline}, the number of points on exactly one $\C$-line is $(q^2+q)q=q^3+q^2$. 
By Lemma~\ref{cplane-meet-theorem}(3), each \infpt\ lies in exactly two
parallel classes, hence there are
${q+1\over 2}$ \infpt s. 
Thus the number of 0-points is $(q^3+q^2+q+1)-({q+1\over 2})-(q^3+q^2)={q+1\over 2}$.
\end{proof}

As a direct consequence of this proof, we have the following corollary which tells us how $\C$-lines are positioned in $\si$.

\begin{corollary} \Label{cline-infty}
If two $\C$-lines meet, then they do so  in an $\infty$-point. Further, each $\infty$-point lies on exactly $2q$ $\C$-lines. 
\end{corollary}

\subsection{Defining a spread in $\si$}

In this section we show how we can use the points of $\C$ to construct a spread in $\si$. We show that each $\C$-point $A$ defines a unique line $t_A$ in $\si$. These resulting $q^2$ lines are mutually skew. The remaining $q+1$ points in $\si$ are the 0-points and \infpt s, we show they lie on a line $\ti$, and that the lines $t_A$ and $\ti$ form a spread in $\si$.
We begin by defining for each $\C$-point $A$, a set of $q+1$ points $t_A$ in $\si$. 

\begin{defn}\Label{def:tA} 
Let $A$ be \cpt\ and let
$\pi$ be a \cpl\ through $A$. The unique tangent to the  conic $(\pi\cap\C)\cup\pi_\infty$ at the point $A$ meets $\si$ in a point which we denote $A_\pi$. Note that $A_\pi$ is on the \cln\ of $\pi$. There are $q+1$ \cpl s $\pi_i$, $i=1,\ldots,q+1$, through $A$. Define $t_A$ to be the set of $q+1$ points $A_{\pi_i}$, $i=1,\ldots,q+1$.
\end{defn}

 We show in Theorem~\ref{cline-cover} that the set $t_A$ is a line in $\si$. First we show that the $\infty$-points and the $0$-points lie on a line of $\si$, denoted $t_\infty$.

\begin{lemma}\Label{thm-tinfty}
The set of \infpt s and 0-points lie on a line $\ti$.
\end{lemma}

\begin{proof}
Let $\ell$ denote the set of \infpt s and 0-points, so by Lemma~\ref{lem:pttype}, 
$|\ell|=q+1$.  Let $\alpha$ be any plane of $\si$.  If $\alpha$ contains
a \cln, then it contains an \infpt\ and so meets $\ell$.  If
$\alpha$ does not contain
a \cln, then $\alpha$ contains at most one point of each of the
$q^2+q$ \cln s of $\si$.  As $\alpha$ has $q^2+q+1$ points, it follows that $\alpha$ meets $\ell$ in at least one point.
Hence every plane of $\si$ meets $\ell$, and so by \cite[Theorem 3.5]{hirs98}, $\ell$ is a line.
\end{proof}



The next two lemmas investigate the 3-space spanned by two $\C$-planes.

\begin{lemma}\Label{cor:old4} A $3$-space contains at most two \cpl s. 
\end{lemma}

\begin{proof}
Let $\alpha$, $\beta$ be two $\C$-planes that span a 3-space, so by Lemma~\ref{cplane-meet-theorem}(2), they are in different parallel classes, and  $\alpha_\infty=\beta_\infty$. 
Thus, for a 3-space to contain three \cpl s, they would
have a common  \infpt, and the planes would be
pairwise in different parallel class. This cannot occur as through any
\infpt\ there are at most two parallel classes by Lemma~\ref{cplane-meet-theorem}(3).
\end{proof}

\begin{lemma}\Label{two-cplanes-3-sp}
Let $\Sigma$ be a $3$-space containing two distinct \cpl s $\alpha$ and $\beta$.
Then the $\C$-points in $\Sigma$ are exactly the $\C$-points in $\alpha$ and $\beta$.
\end{lemma}

\begin{proof}
If two \cpl s $\alpha$ and $\beta$ generate a 3-space $\Sigma$, then they meet in a line  $\ell$, which by Lemma~\ref{cplane-meet-theorem}(2) contains a unique \cpt\ $A$.  
Further, $\alpha,\beta$ are in different parallel classes.
Suppose $\Sigma$ contains a further \cpt\ $B$ not in $\alpha$ or $\beta$. Of the $q+1$ \cpl s through $B$, one is parallel to $\alpha$, one is parallel to $\beta$ and one contains $A$.  Let $\pi$ be any of the $q-2\geq 1$ remaining \cpl s through $B$. Then $\pi$ meets $\alpha$ in a \cpt \ $C$ and $\beta$ in a \cpt\ $D$.  As $B,C,D$ are $\C$-points, they are not collinear, so $\pi=\langle B,C,D\rangle$. Hence $\pi,\alpha,\beta$ are three \cpl s contained in a 3-space $\Sigma$, contradicting Lemma~\ref{cor:old4}.
\end{proof}

We now show that the set of points $t_A$ defined in \ref{def:tA} form a line of $\si$.

\begin{theorem}\Label{cline-cover}
For each \cpt\ $A$, the set $t_A$  is a line of $\si$.
\end{theorem}

\begin{proof}
By Lemma~\ref{lem:pttype}, there are 
 $(q+1)/2$ $\infty$-points, label these  $T_1,\ldots,T_{\frac{q+1}{2}}$. Let $A$ be a $\C$-point, each of the $q+1$ $\C$-planes through $A$ contains exactly one $\infty$-point. If three $\C$-planes $\alpha,\beta,\gamma$ through $A$ contained the same $\infty$-point $T_1$ say, then by Lemma~\ref{cplane-meet-theorem}(2), they are all in different parallel classes. However, this contradicts Lemma~\ref{cplane-meet-theorem}(3) as $T_1$ defines exactly two parallel classes. Hence each of these $q+1$ $\infty$-points lies in exactly two $\C$-planes that contain $A$. Hence the $q+1$ $\C$-planes $\pi_1,\ldots,\pi_{q+1}$ through $A$ can be ordered so that 
the pair $(\pi_1,\pi_2)$ both contain $T_1$, the pair
$(\pi_3,\pi_4)$ both contain $T_2$, and so on.

For each $\C$-plane $\pi_i$, $i=1,\ldots,q+1$, through $A$, define the point $A_{\pi_i}$ as in Definition~\ref{def:tA}, and let $t_A=\{A_{\pi_1},\ldots,A_{\pi_{q+1}}\}$. Now $\pi_1,\pi_2$ both contain the line $AT_1$, so 
$\Sigma=\langle \pi_1,\pi_2\rangle$ is a 3-space. Consider the plane $\sigma_\infty=\Sigma\cap\si$. 
By Lemma~\ref{cor:old4}, $\pi_1$ and $\pi_2$ are the only $\C$-planes contained in $\Sigma$, so the 
remaining $q^2+q-2$ $\C$-planes each meet $\Sigma$ in a line. 
We now show that each of these $q^2+q-2$ lines meet the plane  $\sigma_\infty$ in exactly one point.  Suppose not, that is, suppose that some $\C$-plane $\pi$ meets $\sigma_\infty$ in a line $\ell$, so $\ell$ is the $\C$-line of $\pi$. 
Let $m,n$ be the $\C$-lines of $\pi_1,\pi_2$ respectively. So the plane $\sigma_\infty$ contains three $\C$-lines 
 $\ell,m,n$ which pairwise meet.
By Corollary~\ref{cline-infty}, these $\C$-lines contain a common $\infty$-point, namely $T_1$. As $\pi_1,\pi_2$ meet in a line, by Lemma~\ref{cplane-meet-theorem}(2), $\pi_1,\pi_2$ are in different parallel classes, and $\pi$ is in same parallel class as one of $\pi_1,\pi_2$. Suppose $\pi
$ is in the same parallel class as $\pi_1$, then by  Lemma~\ref{cplane-meet-theorem}, $\pi$ meets $\pi_2$ in a line. 
So as $\ell$ is also in $\sigma_\infty$, we conclude that $\pi$ lies in $\Sigma$, 
contradicting Lemma~\ref{cor:old4}. 
Thus each of  the $q^2+q-1$ \cpl s distinct from  $\pi_1,\pi_2$ contain exactly one point of $\sigma_\infty$.

In particular, as the \cpl s $\pi_3,\ldots,\pi_{q+1}$ all contain $A\in\Sigma$, they meet $\Sigma$ in affine lines labelled $\ell_3,\ldots,\ell_{q+1}$ respectively, so each line $\ell_i$ meets $\si$ in a point. As $\pi_3,\ldots,\pi_{q+1}$ all contain the point $A$, they contain no further \cpt s of $\pi_1$ and $\pi_2$.  In particular, $\ell_i$, $i=3,\ldots,q+1$ does not contain any \cpt s that lie in $\pi_1$ or $\pi_2$ (other than $A$). By Lemma~\ref{two-cplanes-3-sp}, the only $\C$-points in $\Sigma$ are those in $\pi_1$ and $\pi_2$. Hence each $\ell_i$ contains exactly one $\C$-point, namely $A$.  
Recall that $\pi_1$ and $\pi_2$ have a common $\infty$-point $T_1$, and that 
$\pi_3,\ldots,\pi_{q+1}$ do not contain $T_1$. Thus in $\pi_i$, $\ell_i$ contains exactly one $\C$-point $A$, but does not meet $\ti$.
Hence the line $\ell_i$ is the required tangent line to $A$ in $\pi_i$, and hence meets $\sigma_\infty$ in the point $A_{\pi_i}$.

From this we conclude that if $\sigma_\infty$ is the plane defined by the \cln s of $\pi_1$ and $\pi_2$, then $\sigma_\infty$ contains all the points $A_{\pi_i}$ for $i=1,\ldots,q+1$.
Next we repeat the argument with the pair $(\pi_3,\pi_4)$ and define $\sigma'_\infty$ to be the  plane defined by the \cln s of $\pi_3$ and $\pi_4$. A similar argument shows that $\sigma'_\infty$ contains all the points $A_{\pi_i}$ for $i=1,\ldots,q+1$.
As $\sigma_\infty$ and $\sigma'_\infty$ meet $\ti$ in distinct points, it follows that they meet in a line (disjoint from $t_\infty$), and this line $t_A$ contains the $q+1$ points  $A_{\pi_i}$ for $i=1,\ldots,q+1$.
\end{proof}

The next two corollaries follow from the definition of $t_A$, and this result that $t_A$ is a line. 
\begin{corollary}\Label{cor:tA}
Any \cpl\ containing the point $A\in\C$ meets the line $t_A$, and
conversely, any \cpl\ that meets $t_A$ contains the point $A$. 
\end{corollary}

\begin{corollary}\Label{cor:tA2}
 Let $A$ be a $\C$-point, then the affine plane $\langle t_A,A\rangle$ contains exactly one $\C$-point. 
\end{corollary}

We now show that we can use these lines to construct a spread in $\si$. 

\begin{theorem}\Label{lines-form-a-spread}
The lines $t_A$, $A\in\C$, together with $\ti$ form a spread $\S$
of $\Sigma_\infty$.
\end{theorem}

\begin{proof}
Let $\pi$ be a $\C$-plane with $\C$-line $\ell=\pi\cap\si$, and consider the conic $\D=(\C\cap\pi)\cup\pi_\infty$ in $\pi$.
The line $\ell$ is a tangent to $\D$, hence every point of $\ell\setminus\pi_\infty$ lies on one further tangent of $\D$. That is, each point of $\ell\setminus\pi_\infty$ lies in exactly one line $t_A$. By Lemma~\ref{lem:pttype}, each point of $\si\bs t_\infty$ lies in exactly one $\C$-line, hence each point of $\si\bs t_\infty$ lies in exactly one of the $t_A$. Further, the number of points in $\si\bs t_\infty$ equals the number of points on the $q^2$ lines $t_A$. Hence the lines $t_A$, $A\in\C$ and $t_\infty$ are mutually disjoint, and form a spread of $\si$.
\end{proof}

\subsection{The spread $\S$ is regular}

In this section we will show that the spread
$\S=\{t_A\st A\in\C\}\cup\{\ti\}$  is regular. 
We begin with two lemmas showing how certain affine planes meet $\C$.

\begin{lemma}\Label{l-meets-Pinfty}
Let $\ell$ be a line of $\sinfty$ that meets $\ti$, but is not a
\cln. Then 
\begin{enumerate}
\item Every affine plane containing $\ell$ meets $\C$ in at most two
points. 
\item If $\ell$ meets a spread line   $t_A$ (with corresponding $\C$-point $A$), then  the plane $\langle A,\ell\rangle$ contains exactly one \cpt, namely $A$.
\end{enumerate}
\end{lemma}

\begin{proof}
Let $\ell$ be a line of $\sinfty$ that meets $\ti$ in a point, but is not a
\cln. Suppose $\pi$ is an affine plane through $\ell$ that contains three
$\C$-points $P,Q,R$. By (A2), 
the points $P,Q$ lie on a unique \cpl, denoted $\pi_{PQ}$. Similarly
we have \cpl s $\pi_{PR},\pi_{QR}$. 
As $\ell$ is not a \cln, $\pi$ is not a $\C$-plane, hence the  three \cpl s $\pi_{PQ},\pi_{PR},\pi_{QR}$ are distinct.
Consider the 3-space $\Sigma=\langle t_\infty,\pi\rangle$. Now $P,Q,\pi_{PQ}\cap t_\infty$ are three non-collinear points in $\pi_{PQ}$ and in $\Sigma$, hence $\pi_{PQ}\subset\Sigma$. Similarly $\pi_{PR},\pi_{QR}\subset\Sigma$.
This
contradicts Lemma~\ref{cor:old4}. Hence every affine
plane through $\ell$ meets $\C$ in at most two points, proving part 1.

The line $\ell$ meets $\ti$ and  $q$ other lines of the spread, let $t_A$ be one such line, and let $A$ be the $\C$-point corresponding to $t_A$. Consider the plane $\langle A,\ell\rangle$, and suppose it
contains a second point $R$ of $\C$. By (A2), there is a unique \cpl\ $\pi_{AR}$
through $A,R$. The line $AR$ meets $\ell$ in a point $X$. By Corollary~\ref{cor:tA2}, $X\notin t_A$, and the line through $A$ and $\ell\cap\ti$ is a 1-secant of $\C$,
so $X$  lies on a
unique spread line $t_B$, $B\neq A$. Hence the $\C$-plane $\pi_{AR}$ meets
$\ell$ in this point $X$, that is,  the \cln\ $m=\pi_{AR}\cap\sinfty$
meets $\ell$ in the point $X$ of $t_B$. Further $\ell\ne m$ as $m$ is a \cln\ but $\ell$ is not, so 
$\ell$ and $m$ meet $\ti$ in distinct points.
Also note that as $A\in\pi_{AR}$, we have by Corollary~\ref{cor:tA} that
$\pi_{AR}$ meets $t_A$, that is, $m$ meets $t_A$.
Hence $\langle\ell,m\rangle$ is a plane that contains the spread lines $\ti,t_A$, a contradiction as $\ti,t_A$ are skew by Theorem~\ref{lines-form-a-spread}. 
Hence we have shown that $\langle A,\ell\rangle$ contains only one \cpt, namely $A$.
\end{proof}

The next two lemmas examine planes that contain $t_\infty$.

\begin{lemma}\Label{lemma-plane-thru-tinfty}
Each affine plane through $\ti$ meets $\C$ in exactly one point. 
\end{lemma}

\begin{proof}
Suppose  $\alpha$ is a plane through $\ti$ that meets $\C$ in two points
$A,B$. By (A2), there is a unique $\C$-plane $\pi_{AB}$ containing $A$ and $B$.  Let $\pi_{AB}\cap \ti=Y$, then  $Y$ is the $\infty$-point of $\pi_{AB}$,  and so $Y$ 
 can be added to $\pi_{AB}\cap\C$ to
form a conic. This is a contradiction as the line $AB$ meets this conic
in three points, namely $A,B,Y$. Hence each affine plane through $\ti$ meets
$\C$ in at most one point. As there are $q^2$ affine planes through
$\ti$ and $q^2$ points of $\C$, we have exactly one \cpt\ in
each affine plane about $\ti$. 
\end{proof}

\begin{lemma}\Label{cplane-parallel-classes}
Every plane in $\sinfty$ containing $\ti$ contains $q$
\cln s through a common \infpt, and the corresponding $q$ \cpl s are in one parallel class. 
\end{lemma}

\begin{proof}
Consider the $q+1$ planes of $\si$ through $\ti$. 
By Corollary~\ref{cline-infty}, 
each plane about $\ti$ contains at most $q$ \cln s through a common \infpt. There are $q^2+q$ \cln s, hence each of the $q+1$ planes about $\ti$ contains exactly $q$ \cln s through a common $\infty$-point.

Let $\ell,m$ be two $\C$-lines in the same plane of $\si$ about $\ti$. Let $\alpha,\beta$ be the two \cpl s with \cln s $\ell,m$ respectively. Suppose $\alpha,\beta$ are in different parallel classes, then by Lemma~\ref{cplane-meet-theorem}(2) they meet in a line containing a \cpt\ $A$, and span a 3-space.  Consider any other \cln\ $n$ in the plane $\langle \ti,\ell,m\rangle$.  Now if $\delta$ is the \cpl\ with \cln\ $n$, then $\delta$ cannot be parallel to both $\alpha$ and $\beta$ as $\alpha$ and $\beta$ are not in the same parallel class.  So without loss of generality say $\delta$ meets $\alpha$ in the $\C$-point $B$. 
Consider the 3-space $\langle \alpha,\beta\rangle$. It contains $\ell$ and $m$, and so contains $n$. Further, it contains $B$, so $\delta\subset\langle \alpha,\beta\rangle$. This contradicts Lemma~\ref{cor:old4}. Hence  $\alpha$ and $\beta$ are not in different parallel classes, that is, 
the \cln s in a plane of $\si$ through $\ti$ pass through a common
\infpt, and are contained in \cpl s that belong to a common
 parallel class. 
\end{proof}

We now consider a line of $\si$ meeting $t_\infty$ and $q$ other spread lines, and show that the corresponding $\C$-points lie in a common $\C$-plane.
 
 \begin{lemma}\Label{proj-result}
  If $\ell$ is a line of $\si$ meeting spread lines $\ti,
  t_{A_1},\ldots t_{A_q}$, then the corresponding $\C$-points $A_1,\ldots,A_q$ lie on a common \cpl.
\end{lemma}

\begin{proof}
Let $\ell$ be a line meeting spread lines $\ti,
  t_{A_1},\ldots t_{A_q}$. 
  Let $V=\ell\cap t_\infty$. Let $\pi$ be a $\C$-plane not through $V$, let $A$ be a $\C$-point not in $\pi$ and let $\Sigma=\langle A,\pi\rangle$. 
  We first show that $V\notin\Sigma$. 
By Lemma~\ref{lemma-plane-thru-tinfty}, the plane $\langle
\ti,A\rangle$ contains exactly \cpt, namely $A$. As each
affine line of $\pi$ through $\pi_\infty=\pi\cap t_\infty$ meets $\C$ in a point, $\langle
\ti,A\rangle$ meets the plane $\pi$ in exactly the point $\pi_\infty$. Hence
$\langle
\ti,A\rangle$ is not a plane of $\Sigma$, and so $\ti$ is not
contained in $\Sigma$, thus $V\notin\Sigma$. 
Hence we can project the points of $\C$ from $V$ onto $\Sigma$ to obtain a set $\C'$. 
By Lemma~\ref{lemma-plane-thru-tinfty}, 
no line through $V$ can contain 2 points of $\C$, hence $\C'$
contains $q^2$ distinct points.  
We consider this   projection in the two cases: when $V$ is a $0$-point, and when $V$ is an $\infty
$-point. 

Case 1: Suppose $V$ is a $0$-point, we will show that the set $\E=\C'\cup\{\pi_\infty\}$ is an elliptic quadric in
  $\Sigma$.
 Further, each \cpl\ projects to a unique affine plane of $\Sigma$ through $\pi_\infty$, and conversely, each affine plane of $\Sigma$ through $\pi_\infty$ is the image of a unique \cpl.
  
We first show that $\C'\cup\{ \pi_\infty\}$ is a cap of $\Sigma$. By
Lemma~\ref{l-meets-Pinfty}, a plane
meeting $\ti$ in the 0-point $V$ meets $\C$ in at most 2 points. Hence
 no three points of $\C'$ are collinear (otherwise their preimages would
be coplanar with $V$).  
Further, if the point $\pi_\infty$ were collinear with two points $B',C'$ of $\C'$, 
then their preimages $B,C\in\C$ would lie in a plane about $\ti$, which
is not possible by Lemma~\ref{lemma-plane-thru-tinfty}.
So $\E=\C'\cup\{\pi_\infty\}$ is a set of $q^2+1$ points, no three
collinear in a 3-space of order $q$, $q$ odd. Hence $\E$ is an elliptic
quadric, see~\cite{barl55}
or~\cite{pane55}.  
Now let $\alpha$ be a \cpl. As $V$ is a 0-point, $\alpha$ meets
$\ti$ in a point distinct from $V$, and so $\alpha$ projects to an affine plane $\alpha'$ through $\pi_\infty$.  
Let $\alpha,\beta$ be two distinct $\C$-planes with images $\alpha',\beta'$ respectively. If $\alpha'=\beta'$, then as $\alpha,\beta$ together contain at least $2q-1$ $\C$-points, the plane $\alpha'$ contains $2q-1$ points of the elliptic quadric $\E$, a contradiction. 
As there are $q^2+q$ \cpl s and $q^2+q$ affine planes in $\Sigma$
through $\pi_\infty$ (that is, planes not contained in $\si$),  each \cpl\ projects to a unique affine plane of $\Sigma$ through $\pi_\infty$, and conversely. This completes the proof of the statement at Case 1.

 Recall that $\ell$ meets the spread lines $t_{A_1},\ldots,t_{A_q}$,
  we now consider the corresponding $\C$-points $A_i$, and their images $A_i'$ under the projection from $V$.
  By Lemma~\ref{l-meets-Pinfty}, the plane $\pi_i=\langle\ell,A_i\rangle$
contains exactly one \cpt, namely $A_i$. 
Let $L=\ell\cap\Sigma$, so $L\in\si$ as $\ell\subseteq\sinfty$.
The plane $\pi_i=\langle\ell,A_i\rangle$ projects to the line $LA_i'$ which
contains exactly one point of $\C'$, namely $A_i'$, hence these lines are
distinct for distinct $i$. 
As $\sinfty$ contains no \cpt s, the projection $\Sigma\cap\sinfty$ of $\sinfty$ from $V$ onto $\Sigma$ is a tangent plane to $\E$, so $L\pi_\infty$ is a tangent line  to $\E$.  
Thus from a point $L\in\Sigma$ we have a set $\{L\pi_\infty,LA_1',\ldots,LA_q'\}$ of $q+1$ tangent lines of $\E$.  As $\E$ is an
elliptic quadric, the points $\pi_\infty,A_1',\ldots,A_q'$ all lie on a plane
$\beta'$, namely the polar plane of $L$, see \cite[Theorem
15.3.10]{hirs85}. 
Hence by the above argument, $\beta'$ is the image of a \cpl\ $\beta$, and hence the points
$A_i$ lie on a common \cpl, namely $\beta$. So the lemma holds in the case $V=\ell\cap t_\infty$ is a $0$-point. 

Case 2: Suppose $V$ is an \infpt. We will show that the set $\C'$ can be completed to a hyperbolic quadric $\H$ with the
  addition of two lines through $\pi_\infty$ in $\Sigma\cap\Sigma_\infty$.
Further, the \cpl s are of two types: the \cpl s through $V$
  project to the generator lines of $\H$ not through $\pi_\infty$; and the
  \cpl s not through $V$ project to planes through $\pi_\infty$ which meet $\H$ in a conic.
  
  A \cpl\ through $V$ is projected onto a line of $\Sigma$ not through $\pi_\infty$.  By Lemma~\ref{cplane-meet-theorem}(2),  there are two parallel classes ${\mathscr P}_1,{\mathscr P}_2$ of \cpl s through $V$, and \cpl s in the same parallel class pairwise meet in exactly  $V$. Hence one parallel class ${\mathscr P}_1$ through $V$ is mapped to a set   ${\mathcal T}_1$ of $q$ mutually skew lines (each line containing $q$ points of $\C'$), and the other parallel class ${\mathscr P}_2$ through $V$ is mapped to a set ${\mathcal T}_2$ of $q$ mutually skew lines (each line containing $q$ points of $\C'$), with each line from ${\mathcal T}_1$ meeting every  line from ${\mathcal T}_2$ in a point of $\C'$.  We complete these line sets into a regulus and its opposite regulus as follows.  Consider the parallel class ${\mathscr P}_1$ of  \cpl s through $V$.
By Lemma~\ref{cplane-parallel-classes},
 the \cln s of these \cpl s 
lie in a common plane of $\si$ through $\ti$. 
Hence they all
  meet $\Sigma\cap\Sigma_\infty$ in collinear points on a line $\ell_1$
  through $\pi_\infty$.  
The line $\ell_1$ meets every \cpl\ in ${\mathscr P}_1$, and so
meets every line in ${\mathcal T}_1$.
Similarly we have a line $\ell_2$ through $\pi_\infty$ corresponding to the
   parallel class ${\mathscr P}_2$, and $\ell_2$ meets every line in ${\mathcal T}_2$.  Thus the lines
  ${\mathcal T}_1\cup\ell_2$ form a regulus with opposite regulus
  ${\mathcal T}_2\cup\ell_1$.  Hence $\C'\cup\ell_1\cup\ell_2$ is a hyperbolic
  quadric in $\Sigma$.
So we have shown that the $2q$ \cpl s through $V$ project to $2q$ lines of
the hyperbolic quadric $\H$.  The remaining $q^2-q$ \cpl s not through
$V$ project to planes through $\pi_\infty$ that contain $q$ points of $\C'$, hence meet $\H$ in a conic.  
Note that the remaining $2q+1$ planes through $\pi_\infty$ meet $\H$ in two lines of
$\H$, with one of the lines necessarily a line through $\pi_\infty$. This proves the statement for Case 2.

Recall that $\ell$ meets the spread lines $t_{A_1},\ldots,t_{A_q}$,
  we now consider the corresponding $\C$-points $A_i$, and their images $A_i'$ under the projection from $V$. If $\ell$ is a \cln, then
the  $\C$-points $A_1,\ldots,A_q$ lie on a common \cpl\ by Corollary~\ref{cor:tA}.
So suppose $\ell$ is not a \cln. 
By Lemma~\ref{l-meets-Pinfty}, the plane $\pi_i=\langle\ell,A_i\rangle$
contains exactly one \cpt, namely $A_i$. So $\pi_i$ is mapped to a
line through the
points $A_i'$ and $L=\ell\cap(\Sigma\cap\Sigma_\infty)$.  Note that $LA_i'$
meets $\C'$ in exactly one point, namely $A_i'$. The lines
$LA_1',\ldots,LA_q'$ are distinct tangent lines to $\H$.  
Further, $L\pi_\infty$ is a tangent line to $\H$.
Thus from a point $L\not\in\H$ we have a set of $q+1$ tangent lines of $\H$.
Hence the points $\pi_\infty,A_1',\ldots,A_q'$ all lie on a
plane through $\pi_\infty$ which meets $\H$ in a conic, namely the polar plane of
$L$, see~\cite[Theorem 15.3.16]{hirs85}.
Hence by the above argument, this
plane is the image of a \cpl, and hence the points $A_i$ lie
on a common \cpl. That is, the lemma also holds in the case when $V$ is an $\infty$-point.
\end{proof}

We now show that any regulus containing $\ti$ and two
other lines of the spread $\S$ is contained in $\S$. Next we will use the Klein quadric to show that a spread with this property is regular.

\begin{lemma}\Label{partial-spread}
Let $t_A,t_B$ be two elements of the spread $\S$. Then the unique regulus
determined by the three lines $\ti,t_A,t_B$ is contained in $\S$. 
\end{lemma}

\begin{proof}Let $t_A,t_B$ be two elements of the spread $\S$.
Through each point $V_i\in \ti$, $i=0,\ldots,q$, there is a unique line $\ell_i$ that
meets both $t_A$ and $t_B$. Further, the lines $\ell_i$ form a regulus
$\R'$. The opposite regulus $\R$ is the unique regulus containing
$\ti,t_A,t_B$. We want to show that $\R\subset \S$. 

Consider the $\C$-points $A,B$ corresponding to $t_A,t_B$ respectively. By (A2), they lie in a
unique \cpl\ $\pi$.  Now $\pi$ meets $\si$ in a \cln\ 
$\ell$, and $\ell$ meets $\ti$ by Lemmas~\ref{cplane-meet-theorem}(1) and~\ref{thm-tinfty}. Further $\ell$ meets
$t_A$ and $t_B$ by Corollary~\ref{cor:tA}. Hence $\ell$ is one of the lines
$\ell_i\in\R'$. Now $\ell$ meets $q+1$ lines of the spread $\S$, denote them
$\ti,t_A,t_B,t_{C_3},\ldots,t_{C_q}$. We want to show that these are the lines of $\R$. Note that the corresponding $\C$-points $A,B,C_3,\ldots,C_q$ lie on
the \cpl\ $\pi$ by Corollary~\ref{cor:tA}.

Now consider the line $\ell_j\in\R'$,
$\ell_j\neq\ell$, it meets $q+1$
lines of the spread $\S$, denote these by $\ti,t_A,t_B,t_{D_3},\ldots,t_{D_q}$. By
Lemma~\ref{proj-result}, the corresponding $\C$-points
$A,B,D_3,\ldots,D_q$ lie on a common \cpl, $\alpha$ say. As there is a unique
\cpl\ containing $A,B$, we have $\alpha=\pi$, and so
$\{C_3,\ldots,C_q\}=\{D_3,\ldots,D_q\}$. Hence each line $\ell_i$ in the regulus $\R'$ meets the
spread lines $\ti,t_A,t_B,t_{C_3},\ldots,t_{C_q}$, and so $\R=\{\ti,t_A,t_B,t_{C_3},\ldots,t_{C_q}\}$, and so $\R\subset\S$. That is, the unique regulus containing $\ti,t_A,t_B$ is
contained in $\S$.
\end{proof}


To prove that $\S$ is a regular spread
we now show that a spread satisfying the conditions of
Lemma~\ref{partial-spread} is regular. We will use the Klein correspondence
from the set of lines in $\PG(3,q)$ to the set of points of the Klein
quadric $\H_5$, a hyperbolic quadric in $\PG(5,q)$. 
For details of this correspondence, see \cite[Section 15.4]{hirs85}. We
note that a
regular spread of $\PG(3,q)$ corresponds under the Klein correspondence to a 3-dimensional elliptic
quadric contained in $\H_5$, that is, the intersection of $\H_5$ and a
3-space that forms an elliptic quadric. Further, a regulus of $\PG(3,q)$ corresponds to
a conic in $\H_5$. 

\begin{theorem}\Label{Sisregular}
Let $\S$ be a spread of $\PG(3,q)$ with a special line
$\ti$ with the property that the regulus
containing $\ti$ and any further two lines of $\S$ is contained
in $\S$.  Then $\S$ is a regular spread.
\end{theorem}

\begin{proof} We begin with some notation. 
  Let $\S'$ be the points on the Klein quadric $\H_5$ corresponding to the
  lines of the spread $\S$.  For each line $t_A$ of $\S$, let $t_A'$ denote
  the corresponding point of $\H_5$. Similarly, if $\R$ is a regulus of
  $\PG(3,q)$, let $\R'$ denote the set of points of $\H_5$ corresponding
  to the lines of $\R$, and note that the points in $\R'$ lie on a conic.
  For three mutually skew lines $\ell,m,n$  in $\PG(3,q)$, let
  ${\mathscr R}(\ell,m,n)$ denote the unique regulus containing them.
  
  We first show that all the points of $\S'$ are contained in a common
  3-space. We will repeatedly use the assumption that the
  regulus determined by $\ti$ and two other lines of $\S$ is contained
  in $\S$.

Fix a spread element $t_A\ne \ti$ and consider any two distinct reguli
$\R_1,\R_2$ of $\S$ containing $\ti$ and $t_A$.  Suppose that
\begin{eqnarray*}
\R_1&={\mathscr R}(\ti,t_A,t_{B_1})=&\{\ti,t_A,t_{B_1},\ldots,t_{B_{q-1}}\}\\
\R_2&={\mathscr R}(\ti,t_A,t_{C_1})=&\{\ti,t_A,t_{C_1},\ldots,t_{C_{q-1}}\}.
\end{eqnarray*}
Two distinct reguli have at most two common lines, so $\R_1$ and $\R_2$ intersect in
exactly $\ti,t_A$. 
Hence in the Klein quadric, we have two conics $\R_1',\R_2'$, they lie in
two distinct planes which meet
in the line $\ti't_A'$ and hence span a
3-space denoted $\Sigma$. 
Now consider the regulus $\T$ of $\S$ determined by $\ti,t_{B_1},t_{C_1}$: 
\[
\T={\mathscr R}(\ti,t_{B_1},t_{C_1})=\{\ti,t_{B_1},t_{C_1},t_{D_3},\ldots,t_{D_q}\}.
\]
Note that the lines $t_{B_i},t_{C_j},t_{D_k}$ are all distinct.
 In $\PG(5,q)$, $\T'$ 
is a conic 
that contains three points $\ti',t_{B_1}',t_{C_1}'$ of $\Sigma$, hence
$\T'\subset\Sigma$. 

We now use the lines $t_{D_3},\ldots,t_{D_q}$ of $\T$ to construct $q-3+1$ more
reguli of $\S$ through $\ti,t_A$:
\begin{eqnarray*}
\R_i&=&{\mathscr R}(\ti,t_A,t_{D_i}),\quad i=3,\ldots,q.
\end{eqnarray*}
We have a set $\{\R_1,\ldots,\R_q\}$ of $q$ reguli of $\S$ that pairwise intersect in exactly the lines $\ti,t_A$, so
they cover 
$2+q(q-1)=q^2-q+2$ elements of $\S$. The remaining $q-1$ spread elements
$t_{E_1},\ldots,t_{E_{q-1}}$ of $\S$  lie on
a common regulus $\U$ through $\ti,t_A$ (as every three elements determine a unique regulus), that is, $$\U=\{\ti,t_A,t_{E_1},\ldots,t_{E_{q-1}}\}.$$
Now each reguli $\R_i$, $i=3,\ldots,q$, is mapped to a conic $\R_i'$ of
$\H_5$. Further each conic $\R_i'$, $i=3,\ldots,q$ contains three points $\ti',t_A',t_{D_i}'$ of
$\Sigma$. Hence $\R_i'\subset \Sigma$, $i=1,\ldots,q$.  To show that
$\S'\subset\Sigma$, it remains to show that $t_{E_1}',\ldots,t_{E_{q-1}}'\in\Sigma$.

Now consider the two reguli $\R_1$, $\U$ of $\S$.
They map to two conics
$\R_1',\U'$ of $\H_5$ that span a 3-space denoted by $\Sigma'$. Consider
another regulus of $\S$:
\[
\V={\mathscr R}(\ti,t_{B_1},t_{E_1})=\{\ti,t_{B_1},t_{E_1},t_{F_3},\ldots,t_{F_q}\}.
\]
Then $\V'$ is a conic of $\H_5$ with  three points
$\ti',t_{B_1}',t_{E_1}'$ in $\Sigma'$, and so $\V'$ is
contained in $\Sigma'$. As $\U,\V$ meet exactly in $\ti,t_{E_1}$, the
lines $t_{F_i}$, $i=3,\ldots,q$ are distinct from the lines
$t_{E_i}$, $i=1,\ldots,q-1$ and so belong to the $q^2-q+2$ elements of $\S$ in $\Sigma$. Hence we have $t_{F_i}\subseteq\Sigma$ and so $t_{F_i}'\in\Sigma'\cap\Sigma$, $i=3,\ldots,q$. Thus
$\Sigma\cap\Sigma'$ contains $\R_1'$ and $\V'$ which is more than a plane, and so
$\Sigma=\Sigma'$. That is, $t_{E_1}',\ldots,t_{E_{q-1}}'\in\Sigma$.
Thus the lines of $\S$ are mapped into points of $\H_5$ that lie in a
3-space $\Sigma$.

The intersection $\Q$ of the 3-space $\Sigma$ with $\H_5$ is a quadric of $\Sigma$, hence is either an elliptic, hyperbolic or a degenerate quadric.  Note that in all cases other than the elliptic quadric, $\Q$ does not correspond to a spread (or a set containing a spread) of $\PG(3,q)$. Hence $\Q$ is an elliptic quadric,  and so $\S$ is a regular spread.
\end{proof}

As an immediate consequence of Lemma~\ref{partial-spread} and Theorem~\ref{Sisregular}, we have that the spread $\S$ constructed in Theorem~\ref{lines-form-a-spread} is regular.

\begin{corollary}\Label{cor:regular-spread}
 The spread $\S=\{t_A\st A\in\C\}\cup\{\ti\}$ is regular.
\end{corollary}

\subsection{$\C$ gives rise to an arc in $\PG(2,q^2)$}

By Corollary~\ref{cor:regular-spread}, we have a regular spread
$\S=\{t_A\st A\in\C\}\cup\{\ti\}$ in $\sinfty$ from which we can construct a
Desarguesian plane $\P(\S)\cong\PG(2,q^2)$ via the Bruck-Bose
correspondence.  
Let $\Cp$ be the
set of points in $\P(\S)$ corresponding to the affine points of $\C$
together with the point $T_\infty$ on $\li$ corresponding to the spread line $\ti$.

\begin{theorem}\Label{thm:isanarc}
 $\Cp$ is a conic in $\P(\S)\cong\PG(2,q^2)$. 
\end{theorem}

\begin{proof}
We show that $\O=\C\cup T_\infty$ is a $(q^2+1)$-arc in $\PG(2,q^2)$, and hence a conic. 
A line through $T_\infty$ in $\PG(2,q^2)$ corresponds to an affine plane of
$\PG(4,q)$ that contains the spread line $\ti$. By Lemma~\ref{lemma-plane-thru-tinfty}, the affine planes of $\PG(4,q)$  through $\ti$ meet $\C$ in exactly one point.
Hence in $\PG(2,q^2)$, a line through $T_\infty$ meets
$\Cp$ in at exactly one further point.

Let $t_A$ be a line of the spread $\S$ in $\PG(4,q)$ with corresponding $\C$-point $A$. 
Let $\alpha$ be an affine plane of $\PG(4,q)$ through the spread line $t_A$ that contains three $\C$-points $P,Q,R$.     
We obtain a contradiction to show that this is not possible. 
Note that $\alpha\neq\langle A,t_A\rangle$ since 
by Corollary~\ref{cor:tA2}, the plane $\langle A,t_A\rangle$ contains exactly one $\C$-point. 
By (A2), $P,Q$ lie in a unique $\C$-plane $\pi_{PQ}$ which meets $t_A$ in the point $PQ\cap t_A$. 
 Hence  by Corollary~\ref{cor:tA}, $\pi_{PQ}$ contains the point $A$. 
Similarly, $P,R$ lie on a unique \cpl\ $\pi_{PR}$ that contains
$A$. Hence we have two distinct \cpl s $\pi_{PQ},\pi_{PR}$ that both contain the
two distinct points $A,P$ of $\C$, contradicting (A2). Hence any affine plane  that contains a spread line $t_A$ contains at
most two points of $\C$. Thus in the Bruck-Bose plane $\PG(2,q^2)$, a line through a point of
$\linfty\setminus T_\infty$ meets $\Cp$ in at most two points. Further,
$\linfty$ meets $\Cp$ in one point, so we have shown that $\Cp$ is a
$(q^2+1)$-arc in $\PG(2,q^2)$. 
As $q$ is odd, by Segre \cite{segr55}, $\Cp$ is a
conic in $\PG(2,q^2)$.  \end{proof}

This almost completes the proof of Theorem~\ref{mainthm}. It remains to show that the spread $\S$ is unique, which we do in the next section. 

\subsection{The spread $\S$ is unique}

We now show that the spread $\S=\{t_A\st A\in\C\}\cup\{\ti\}$ constructed in Theorem~\ref{lines-form-a-spread} is the only spread in $\si$ for which the $\C$-points give rise to an arc in the Bruck-Bose plane $\P(\S)$.

\begin{theorem}\Label{thm:diffspread}
Let $\S'$ be a spread of $\si$ distinct from the spread $\S=\{t_A\st A\in\C\}\cup\{\ti\}$. Then in the associated
Bruck-Bose plane $\P(\S')$,
the set of points corresponding to $\C$ do not form an arc.
\end{theorem}

\begin{proof}
  Let $\ell$ be a line in $\sinfty$ disjoint from $\ti$ and assume
  that all the affine planes through $\ell$ meet
  $\C$ in at most two points. We will show that $\ell$ must be one of the
  spread lines $t_A$ for some point $A\in\C$. 
Let $X_1$ be a point on $\ell$, by Lemma~\ref{lem:pttype} and ~\ref{thm-tinfty},  there exists a unique \cln\ $\ell_{1}$
through $X_1$. By Corollary~\ref{count-cline}, $\ell_1$ is contained in a unique \cpl\ $\pi_{1}$. Let ${\mathcal E}_1$ be
the conic in $\pi_{1}$, that is,
${\mathcal E}_1=(\C\cap\pi_{1})\cup(\pi_{1}\cap \ti)$. Now $X_1$ is on one tangent to ${\mathcal E}_1$, namely
$\ell_{1}$, hence $X_1$ is on a second tangent $m_{1}$ to ${\mathcal E}_1$.
There are $(q-1)/2$ 2-secants to ${\mathcal E}_1$ through $X_1$, each together with the line $\ell$
determines an affine  plane, we label these planes $\alpha_1,\ldots,\alpha_{(q-1)/2}$. As
each plane contains a 2-secant to ${\mathcal E}_1$ through $X_1$, each contains at least two \cpt s, and hence by our assumption about $\ell$, contains exactly two \cpt s.

Similarly, let $X_2$ be a point on $\ell$, lying on a unique $\C$-line $\ell_2$, defining the unique \cpl\ $\pi_2$, and construct planes $\beta_1,\ldots,\beta_{(q-1)/2}$ which contain $\ell$ and exactly two \cpt s.
Suppose $\alpha=\alpha_i=\beta_j$ for some $i,j$. Then $\alpha$ contains at least three $\C$-points (as $\alpha_i$ contains a 2-secant of $\C$ through $X_1$, $\beta_j$ contains a 2-secant of $\C$ through $X_2$, and $X_1\neq X_2$). This contradicts our assumption that planes about $\ell$ contain at most two $\C$-points. Hence the planes $\alpha_1,\ldots,\alpha_{(q-1)/2}, \beta_1,\ldots,\beta_{(q-1)/2}$ are distinct, and cover $(q-1)\times 2$ distinct $\C$-points.

Repeating this for the remaining points $X_3,\ldots,X_{q+1}$ of $\ell$, we obtain a set $
\mathcal K$ of distinct planes about $\ell$ (including the $\alpha_i,\beta_j$) of size $(q+1)\times (q-1)/2$. Each plane in $\mathcal K$ contains
exactly two \cpt s, accounting for $q^2-1$ points of $\C$ that lie
on planes through $\ell$.

 Now for each $i$, consider the unique affine tangent line $m_{i}$ through $X_i$ (distinct from $\ell_i$) to the conic ${\mathcal E}_i$, in
the plane $\pi_{i}$. We have a set of $q+1$ planes $\langle \ell,m_{i}\rangle$, not necessarily distinct, but distinct from the planes of $\mathcal K$, each
 meeting $\C$ in at least one point. However, there is only one \cpt\ unaccounted for by the planes of $\mathcal K$, so the planes $\langle \ell,m_{i}\rangle$  are all the same plane $\alpha$ which contains exactly one point $Z$ of $\C$. The $q+1$ lines $ZX_1,\ldots,ZX_{q+1}$ through $Z$ in $\alpha$ are respectively the tangents $m_1,\ldots,m_{q+1}$ to the conic ${\mathcal E}_1,\ldots,{\mathcal E}_{q+1}$ in $\pi_{1},\ldots,\pi_{q+1}$ at $Z$.
Further, the tangents $m_1,\ldots,m_{q+1}$ meet $\Sigma_\infty$ at points $X_1,\ldots,X_{q+1}$ of $\ell$, so by Definition~\ref{def:tA},  $\ell=t_Z$.
 That is, $\ell$ is one of the spread lines of $\S$. We conclude
  that every line in $\sinfty$ which is disjoint from $\ti$ and not in $\S$ lies on an affine plane
  that contains at least three \cpt s. 

Suppose $\S'$ is a spread of $\si$ distinct from $\S$ such that in the Bruck-Bose plane $\P(\S')$ the $\C$-points form an arc. Let $\S'\setminus\S$ denote the set of lines in $\S'$ that are not in $\S$. 
If $\S'\setminus\S$ contains a line $\ell$ disjoint from $t_\infty$, then by the above argument, $\ell$ lies on some affine plane $\alpha$ that contains at least three $\C$-points. In the Bruck-Bose plane $\P(\S')$, $\alpha$ corresponds to a line that contains three $\C$-points, so $\C$ is not an arc, contradicting our assumption. 
Hence $\S'\setminus \S$ cannot contain a line
$\ell$ disjoint from $\ti$,
thus the lines in $\S'\setminus \S$ must meet $\ti$. 
So if $\S'\neq\S$, then $\S'\setminus\S$ contains $q+1$ lines that meet $t_\infty$ in a point. 
Let $\ell\in\S'\setminus\S$, so $\ell$ meets $\ti$ and another $q$ spread lines $t_{A_1},\ldots,t_{A_q}$ of $\S$. Let  $\R=\{ \ti,t_{A_1},\ldots,t_{A_q}\}$, then  the lines of $\S'\setminus\S$ cover the same points as the lines of $\R$ do.  Let $\S'\setminus\S=\{\ell,\ell_1,\ldots,\ell_q\}$, then each line $\ell_i$ meets each line in $\R$.

Let $t_A$, $t_B$ be any two lines of $\R$ distinct from $t_\infty$.
There are exactly $q+1$ lines meeting $\ti,t_A,t_B$, and since $\ell,\ell_1,\ldots,\ell_q$ meet each of $\ti,t_A,t_B$, they form a regulus $\R'$.  A similar argument with three elements of $\R'$ shows that $\R$ is the opposite regulus of $\R'$.
By
Corollary~\ref{cor:tA}, there exists a unique
\cln\ $m$ meeting  $\ti,t_A,t_B$, namely the \cln\ 
corresponding to the unique \cpl\ $\pi_{AB}$ through $A$ and $B$. 
The line $m$ meets three lines of $\R$ and hence it is a line of
$\R'$. 
That is, $m$ is a line of $\S'$. However, the $\C$-plane $\pi_{AB}$ through $m$ corresponds to a line of $\P(\S')$ that contains $q$ points of $\C$. Hence $\C$ is not an arc of $\P(\S')$. 
Thus   the only spread which gives rise to an arc in the corresponding Bruck-Bose  plane is the regular spread $\S$ constructed in Theorem~\ref{lines-form-a-spread}.
\end{proof}

This completes the proof of Theorem~\ref{mainthm}.

\section{Conclusion}

In this paper we characterised sets in $\PG(4,q)$, $q$ odd, $q\geq 7$ satisfying the combinatorial properties given in 
Theorem~\ref{mainthm}
as corresponding via the Bruck-Bose correspondence to conics in $\PG(2,q^2)$. 
We note that a similar characterisation when $q$ is even is given in \cite{conicqeven}.
The cases when $q=3$ or $5$ are still open.

An interesting geometric question arises from the properties of a conic given in Lemma~\ref{conic-satisfies-props}. Let $\C$ be a conic in $\PG(2,q^2)$, $q$ odd, tangent to $\li$, and let $\pi$ be a $\C$-plane. By property 3 of Lemma~\ref{conic-satisfies-props}, the points of $\pi$ that are not in $\C$ lie on exactly one more $\C$-plane. Let  $P$ be a point of $\pi\setminus\C$. If $P$ is an interior point of the subconic $\pi_\C=\pi\cap\C$, then we can use the polarity of $\pi_\C$ to construct the second $\C$-plane containing $P$. If $P$ is an exterior point of $\pi_\C$, then it would be interesting to have a geometric construction of the second $\C$-plane containing $P$.

\end{document}